\documentclass[11pt]{article}
\usepackage[a4paper,margin=1.1in]{geometry}
\usepackage[colorlinks,citecolor=magenta,linkcolor=black]{hyperref}
\pdfpagewidth=\paperwidth \pdfpageheight=\paperheight
\usepackage{amsfonts,amssymb,amsthm,amsmath,eucal,tabu,url}
\usepackage{pgf}
 \usepackage{array}
 \usepackage{tikz-cd}
 \usepackage{pstricks}
 \usepackage{pstricks-add}
 \usepackage{pgf,tikz}
 \usetikzlibrary{automata}
 \usetikzlibrary{arrows}
 \usepackage{indentfirst}
\usepackage{tabularx} 

\theoremstyle{plain}
\newtheorem{thm}{Theorem}[section]
\newtheorem{theorem}[thm]{Theorem}

\newtheorem{lemma}[thm]{Lemma}
\newtheorem{proposition}[thm]{Proposition}
\newtheorem{corollary}[thm]{Corollary}
\newtheorem{conjecture}[thm]{Conjecture}

\theoremstyle{definition}
\newtheorem{definition}[thm]{Definition}
\newtheorem{remark}[thm]{Remark}
\newtheorem{example}[thm]{Example}

\newtheorem{problem}[thm]{Problem}

\newtheorem{thevarthm}[thm]{\varthmname}

\newenvironment{varthm*}[1]{\trivlist\item[]{\bf #1.}\it}{\endtrivlist}

\renewcommand\geq{\geqslant}

\renewcommand\leq{\leqslant}

\newcommand\be{\begin{eqnarray*}}
\newcommand\ee{\end{eqnarray*}}

\newcommand\newop[2]{\def#1{\mathop{\rm #2}\nolimits}}
\newop\log{log}
\newop\ord{ord}
\newop\Gal{Gal}
\newop\SL{SL}
\newop\Bl{Bl}
\newop\mult{mult}
\newop\mass{mass}
\newop\div{div}
\newop\codim{codim}
\newop\sing{sing}
\newop\vdim{vdim}
\newop\edim{edim}
\newop\Ass{Ass}
\newop\size{size}
\newop\reg{reg}
\newop\satdeg{satdeg}
\newop\supp{supp}
\newop\Neg{Neg}
\newop\Nef{Nef}
\newop\Nefh{Nef_H}
\newop\Eff{Eff}
\newop\Zar{Zar}
\newop\MB{MB} 

\newop\MBxC{MB\mathit{(x,C)}}
\newop\NnB{NnB}
\newop\Bigg{Big}
\newop\Effbar{\overline{\Eff}}

\def\keywordname{{\bfseries Keywords}}%
\def\keywords#1{\par\addvspace\medskipamount{\rightskip=0pt plus1cm
\def\and{\ifhmode\unskip\nobreak\fi\ $\cdot$
}\noindent\keywordname\enspace\ignorespaces#1\par}}
\def\subclassname{{\bfseries Mathematics Subject Classification
(2020)}\enspace}
\def\subclass#1{\par\addvspace\medskipamount{\rightskip=0pt plus1cm
\def\and{\ifhmode\unskip\nobreak\fi\ $\cdot$
}\noindent\subclassname\ignorespaces#1\par}}

\begin{document}
\title{Algebraic properties of Levi graphs associated with curve arrangements}
\author{Piotr Pokora and Tim R\"omer}
\date{}
\maketitle
{\centering\footnotesize To J\"urgen Herzog on his 80th birthday\par}
\thispagestyle{empty}

\begin{abstract}
In the present paper we study algebraic properties of edge ideals associated with plane curve arrangements via their Levi graphs. Using combinatorial properties of such Levi graphs we are able to describe those monomial algebras being Cohen-Macaulay, Buchsbaum, and sequentially Cohen-Macaulay. We also consider the projective dimension and the Castelnuovo-Mumford regularity for these edge ideals. We provide effective lower and upper bounds on them. As a byproduct of our study we connect, in general, various Buchsbaum properties of squarefree modules.
\subclass{05E40; 13F55 (Primary); 05C25; 13H10; 14N20 (Secondary)}

\end{abstract}

\section{Introduction}
In the present paper we study edge ideals associated with plane curve arrangements in the complex projective plane. Such ideals can be naturally defined via the notion of a Levi graph that encodes properties of the intersection poset of an arrangement. The goal is to understand how graph theoretic invariants are encoded algebraically in these ideals. Knowing the linkage between  Levi graphs and the geometry of arrangements, we are able to understand many algebraic properties of associated monomial algebras via curve arrangements and their combinatorics. This approach is the main novelty of our present work. 

Our choice of Levi graphs as a potential source for the poset structure for the associated monomial algebras is not accidental and it is motivated by numerous applications. In the context of line arrangements, Levi graphs play a leading role in many subjects of current research. These graphs, which were introduced by Coxeter \cite{Coxeter}, are the bipartite graphs that decode the intersection poset structure of a given arrangement of lines $\mathcal{L}\subset \mathbb{P}^{2}_{\Bbbk}$. More precisely, for a collection of points and lines in a projective plane, we construct a graph with one vertex per point, one vertex per line, and an edge for every incidence between a point and a line. 

It turns out that a Levi graph carries the same combinatorial information as the intersection lattice, which is the reason why this object attracts a lot of attention. One of the most fundamental problems that involves line arrangements in the plane is the celebrated Terao's conjecture on freeness of hyperplane arrangements, i.e., the freeness of the module of polynomial derivatives associated to an arrangement $\mathcal{L}$ is determined by its Levi graph (or equivalently, the intersection lattice). In the case of arbitrary arrangements of curves in the plane, the intersection structure, in general, does not form a lattice, but only the intersection poset. However, in some cases we can still define honest variant of the Levi graph that will decode the intersection structure, and we are going to study it in the context of $d$-arrangements of plane curves, a special class of smooth plane curves that naturally generalizes line arrangements in the plane.

Edge ideals can be studied from the viewpoint of Stanley-Reisner rings, where the poset structure of the facets associated with the abstract simplicial complex is the crucial information that tells us about many algebraic properties. There are many interesting articles and books devoted to properties of edge ideals like \cite{vanTuyl2,VillarrealB}. Here we are going to study edge ideals defined via Levi graphs of arrangements and we want to understand algebraic properties of those ideals by looking at the combinatorics of Levi graphs. From this perspective, we will be able to provide complete classification results which is the key advantage of our approach. More precisely, the main purpose of the present manuscript is to introduce a certain class of edge ideals determined by the intersection posets of $d$-arrangements via the notion of their Levi graph. This will allow us to reinterpret algebraic properties of monomial algebras in the language of the geometry of arrangements. For instance, we will be able to say which intersection posets of $d$-arrangements of curves will provide to us examples of sequentially Cohen-Macaulay algebras, see Theorem \ref{theoremb}. These results might be used in different branches of algebraic geometry, where the presence of edge ideals and the associated bipartite graphs plays an prominent role. In the application context, it is worth mentioning that edge ideals arise naturally in the context of rigid toric varieties determined by bipartite graphs \cite{Portakal}. Except results devoted to edge ideals and arrangements of curves, we study general properties of squarefree modules and their homological properties. In Section \ref{seq:s4} a special class of dual squarefree modules is defined and prove that for squarefree $S = \Bbbk[x_{1}, ..., x_{n}]$-modules being Buchsbaum is equivalent to be $k$-Buchsbaum for every $k\geq 1$, see Theorem \ref{thm:buchsbaum-characterizations}. This general result allows us to conclude that for edge ideals $I(G)$ associated with bipartite graphs $G$ the condition that the quotient algebra $S/I(G)$ is Cohen-Macaulay is  equivalent to be both Buchsbaum and $k$-Buchsbaum for every $k \geq 1$. At the end of the paper we study homological properties of the quotient algebras $S/I(G)$, where $I(G)$ is the edge ideal associated with the Levi graph of a $d$-arrangement, and we provide some effective bounds on the projective dimension and the Castelnuovo-Mumford regularity for powers of edge ideals. Based on examples, several open problems related to the geometry of $d$-arrangements and homological properties of the associated edge ideals are formulated. For some examples, that allow us to made our predictions, we performed symbolic computations supported by scripts written in \verb{Singular{ \cite{Singular}.

\section{Arrangements of curves}
In this section we are going to define our main geometrical object of studies, namely $d$-arrangements of plane curves. Such arrangements can be considered as a natural generalization of line arrangements, especially from the viewpoint of their singularities. Let us emphasize in this place that for us singular points are just the intersection points of curves, i.e., points in the plane where at least two curves from the arrangement meet.
\begin{definition}
Let $\mathcal{C} = \{C_{1}, ..., C_{k}\} \subset \mathbb{P}^{2}_{\mathbb{C}}$ be an arrangement of $k\geq 3$ curves in the plane. We say that $\mathcal{C}$ is a  \emph{$d$-arrangement} if
\begin{itemize}
	\item all curves $C_{i}$ are smooth of the same degree $d\geq 1$;
	\item the singular locus ${\rm Sing}(\mathcal{C})$ consists of only ordinary intersection points -- they look locally like intersections of lines.
\end{itemize}
\end{definition}
In particular, $1$-arrangements are just line arrangements. However, the class of $2$-arrangement is strictly smaller than the class of all conic arrangements in the plane -- usually the intersection points of conic arrangements \textbf{are not ordinary}. If $\mathcal{C}$ is a $d$-arrangement, then we have the following combinatorial count
\begin{equation}
    d^{2} \binom{k}{2} = \sum_{p \in {\rm Sing}(\mathcal{C})} \binom{m_{p}}{2},
\end{equation}
where $m_{p}$ denotes the multiplicity, i.e., the number of curves from $\mathcal{C}$ passing through $p \in {\rm Sing}(\mathcal{C})$. Moreover, for every $C_{j} \in \mathcal{C}$ one has
\begin{equation}
\label{equ2}
d^{2}(k-1) = \sum_{p \in {\rm Sing}(\mathcal{C}) \cap C_{j}} (m_{p}-1).
\end{equation}
Using Hirzebruch's convention, for a given $d$-arrangement $\mathcal{C}$ we denote by $t_{r} = t_{r}(\mathcal{C})$ the number of $r$-fold points, i.e., points in the plane where exactly $r$ curves from the arrangement meet. Additionally, 
for $s$ being the number of intersection points in ${\rm Sing}(\mathcal{C})$ we also have
$$
s = \sum_{r\geq 2} t_{r}.
$$

One of the most important information that each $d$-arrangement carries is the intersection poset which decodes the incidences between curves and intersection points. Since all the intersection points of $d$-arrangements are ordinary, they behave analytically like line arrangements around the intersection points. This observation motivates us to introduce the notion of the Levi graph for $d$-arrangements.

\begin{definition}
Let $\mathcal{C}$ be a $d$-arrangement. Then the \emph{Levi graph} $G = (V,E)$ is a bipartite graph with  $V : = V_{1} \cup V_{2} = \{x_{1}, ..., x_{s}, y_{1}, ..., y_{k}\}$, where each vertex $y_{i}$ corresponds to a curve $C_{i}$, each vertex $x_{j}$ corresponds to an intersection point $p_{j} \in {\rm Sing}(\mathcal{C})$ and $x_{j}$ is joined with $y_{i}$ by an edge in $E$ if and only if $p_{j}$ is incident with $C_{i}$.  
\end{definition} 

The importance of the Levi graphs for $d$-arrangements can be observed in the context of a naive generalization of Terao's conjecture on the freeness of $d$-arrangements. For the completeness of the exposition, let us present this conjecture.
\begin{conjecture}[Generalized Terao conjecture]
\label{Conj:Generalized-Terao-conjecture}
Let $\mathcal{C}_{1},\mathcal{C}_{2} \subset \mathbb{P}^{2}_{\mathbb{C}}$ be two $d$-arrangements for a fixed $d\geq 1$ and denote by $G_{1}, G_{2}$ the associated Levi graphs. Assume that $G_{1}$ and $G_{2}$ are isomorphic, is it true that the freenees of $\mathcal{C}_{1}$ implies that $\mathcal{C}_{2}$ is also free?
\end{conjecture}
Our expectation is that this generalization of Terao's conjecture should be false for $d\geq 2$, mostly due to the fact that there are examples of conic-line arrangements in the plane with ordinary singularities providing counterexamples to a generalized Terao's conjecture, see \cite{Schenck} for details. 

\section{Edge ideals associated with $d$-arrangements}

Let $G$ be a simple graph with vertices $x_{1}, ..., x_{n}$ and let $S = \Bbbk[x_{1}, ..., x_{n}]$ be a polynomial ring over a fixed field $\Bbbk$. We will use the convention that $x_{i}$ will denote both a vertex of $G$ and also a variable of $S$. 
\begin{definition}
The \emph{edge ideal} $I(G)$ associated with the graph $G$ is the ideal of $S$ generated by the set of all squarefree monomials $x_{i}x_{j}$ such that $x_{i}$ is adjacent to $x_{j}$. The ring $S/I(G)$ is called the \emph{edge ring} of $G$.
\end{definition}
In the setting of Levi graphs associated with $d$-arrangements, we will work in $S = \Bbbk[x_{1}, ..., x_{s},y_{1}, ..., y_{k}]$ in order to emphasize the incidence structure of the Levi graph $G$ of $\mathcal{C}$, namely $x_{i}$'s correspond to the intersection points, $y_{j}$'s correspond to curves in a $d$-arrangement, and we have an edge joining $x_{i}$ and $y_{j}$ if and only if the intersection point $P_{i}$ lies on the curve $C_{j}$ in $\mathcal{C}$. 

\begin{definition}
Given a poset $\mathcal{P}$ with the vertex set $X = \{x_{1}, ..., x_{n}\}$, its \emph{order complex}, denoted by $\triangle$, is the simplicial complex on $X$ whose faces are the chains (linearly ordered sets) in $\mathcal{P}$. 
\end{definition}
With the notation of the definition we have that
$$
\Bbbk[\triangle] = \Bbbk[X] / \langle x_{i}x_{j} \, : \, x_{i} \nsim x_{j} \rangle$$
is the Stanley-Reisner ring of $\triangle$, where $x_{i} \nsim x_{j}$ means that the elements $x_{i}$ and $x_{j}$ are not comparable.
The \emph{simplicial complex of the graph $G$} is defined by
$$\triangle_{G} = \{ A \subseteq V \, : \, A \text{ is an independent set in } G\},$$
where $A$ is an independent set in $G$ if none of its elements are adjacent. Observe that $\triangle_{G}$ is precisely the simplicial complex with the Stanley-Reisner ideal $I(G)$. Recall:

\begin{definition}
Let $I \subset S$ be a graded ideal. The quotient ring $R/I$ is \emph{Cohen-Macaulay} if ${\rm depth}(R/I) = {\rm dim}(R/I)$. The ideal $I$ is called \emph{Cohen-Macaulay} if $R/I$ is Cohen-Macaulay.
\end{definition}
See the book \cite{BrunsHerzog} as a general reference on Cohen-Macaulay ideals and modules. We say that \emph{a graph $G$ is Cohen-Macaulay} over $\Bbbk$ if $R/I(G)$ is a Cohen-Macaulay ring. 

It is worth noticing also that the celebrated result of Eagon and Reiner states that an abstract simplicial complex $\triangle$ is Cohen–Macaulay if and only if the Stanley–Reisner ideal of its Alexander dual, denoted by $I^{\vee}$, has a linear resolution. We are going to show an example of a Cohen-Macaulay graph $G$ that is associated with a point-line configuration in the plane.
\begin{center}
\textbf{Warning}: The class of point-line configurations in the plane is strictly larger that the class of line arrangements.
\end{center}
For us every line arrangement is a point-line configuration, but the revers statement is obviously false. For instance, for a point-line configuration, a point in the configuration need not to be an intersection point!
\begin{example}
Consider a point-line configuration in the plane consisting of $3$ points and $3$ lines, where we have exactly one triple intersection point and two points which are sitting on two distinct lines - these two points are not intersection points! More precisely, below the Levi graph of this configurations is presented -- here $x_{i}$'s denote the points, and $y_{j}$'s denote the lines:
$$G = \bigg( \{x_{1},x_{2},x_{3},y_{1},y_{2},y_{3}\}, \bigg\{\{x_{1},y_{1}\},\{x_{2},y_{1}\},\{x_{2},y_{2}\},\{x_{2},y_{3}\},\{x_{3},y_{3}\}\bigg\}\bigg).$$
Let 
$$S = \Bbbk[x_{1},x_{2},x_{3},y_{1},y_{2},y_{3}]
$$ and consider the following  edge ideal associated with $G$, namely
$$I(G) = \langle x_{1}y_{1}, x_{2}y_{1},x_{2}y_{2},x_{2}y_{3},x_{3}y_{3} \rangle.$$
We show that the monomial algebra
$$M:= S / I(G)$$
is Cohen-Macaulay. In order to do so, we prove that the Alexander dual $I(G)^{\vee}$ has a linear resolution. First of all, one needs to find the Alexander dual of $I(G)$, and this can be done by computing the primary decomposition of $I$. As a result, we obtain
$$
I(G)^{\vee} = \langle y_{1}y_{2}y_{3}, y_{1}y_{3}x_{2}, y_{3}x_{1}x_{2},y_{1}x_{2}x_{3}, x_{1}x_{2}x_{3}\rangle.
$$
The minimal graded free resolution of $N=S / I(G)^{\vee}$ has the following form
$$0 \rightarrow S(-5) \rightarrow S^{5}(-4) \rightarrow S^{5}(-3) \rightarrow S \rightarrow N \rightarrow 0,$$
which shows that $M$ is a Cohen-Macaulay monomial algebra.
\end{example}

\section{Squarefree modules and their properties}
\label{seq:s4}

The main goal of this section is to recall the notion
of squarefree modules, which were introduced and studied by Yanagawa in \cite{Yanagawa}. See also \cite{Tim1, Tim2} for related results.
We study further aspects of them needed in the following. For this we start with our general setup which follows  \cite[Section 1 and 2]{GotoWatanabe}. All modules in this sections are finitely generated $\mathbb{Z}^n$-graded $S$-modules $M=\oplus_{\textbf{a}\in \mathbb{Z}^n} M_{\textbf{a}}$. Such a module $M$ is called \emph{$\mathbb{N}^n$-graded}, if $M_{\textbf{a}}=0$ for $\textbf{a} \not\in \mathbb{N}^n$. Let $\Bbbk$ be a field and denote by $S=\Bbbk [x_{1}, ..., x_{n}]$ the polynomial ring with the standard $\mathbb{Z}^{n}$-grading and let $\mathfrak{m} = (x_{1}, ..., x_{n})$ be the standard graded maximal ideal. In this section we consider $S$-modules $M$ with $\mathbb{Z}^{n}$-grading as a default grading. For such a module one defines the \emph{dual module} $M^{\vee}$ as
$$N:=M^{\vee}= {\rm Hom}_{\Bbbk}(M, \Bbbk).$$
\begin{remark}
The defined above dual $S$-module $N$ to $M$ can be viewed from a viewpoint of the Matlis duality since we have
$$N = {\rm Hom}_{\Bbbk}(M,\Bbbk) \cong {\rm Hom}_{S}(M,E_{S}(\Bbbk)),$$
where $E_{S}(\Bbbk) = {\rm Hom}_{\Bbbk}(S,\Bbbk) = \Bbbk[x_{1}^{-1}, ..., x_{n}^{-1}]$ is the injective hull of $\Bbbk$ as an $S$-module.
The degree $\textbf{b} \in \mathbb{Z}^{n}$ part of the Matlis duality is
$$N_{\textbf{b}} :=(M^{\vee})_{\textbf{b}} = {\rm Hom}_{\Bbbk}(M_{-\textbf{b}}, \Bbbk),$$
so the Matlis duality reverses the grading in the presented sense. 

Under the assumption that $M_{\textbf{b}}$ is a finitely dimensional $\Bbbk$-vector space for every $\textbf{b} \in \mathbb{Z}^{n}$, we have the isomorphism  $(M^\vee)^\vee \cong M$. 
\end{remark}

Let $\mathcal{M}$ be the category of $\mathbb{Z}^{n}$-graded $S$-modules and whose morphisms are homogeneous of degree $0$. Observe that the multiplication by $x_{i}$ gives the following homomorphism of $\Bbbk$-vector spaces
$$x_{i} : M_{\textbf{b}} \longrightarrow M_{\textbf{b} + \varepsilon_{i}}$$
for any $M = \bigoplus_{\textbf{b} \in \mathbb{Z}^{n}}M_{\textbf{b}} \in \mathcal{M}$, where $\varepsilon_{i}$ is the $i$-th vector of the canonical basis of $\mathbb{Z}^{n}$.

Using the Matlis duality we can define the multiplication map on the duals, namely 
$$x_{i} \, : N_{\textbf{b}} \rightarrow N_{\textbf{b} + \varepsilon_{i}}$$
which has the following explicit form
$(x_{i}\cdot \phi)(m) = \phi(x_{i}\cdot m).$

Finally, let us denote the canonical module $\omega_{S}$ in the $\mathbb{Z}^{n}$-grading situation, i.e., 
$
\omega_{S} = S(-1, ...,-1) = S(-\textbf{1}).$ The  graded local duality theorem \cite[Theorem 2.2.2]{GotoWatanabe} states:

\begin{theorem}
\label{local:dual}
For any finitely generated $\mathbb{Z}^{n}$-graded $S$-module $M$ one has
$$H^{i}_{\mathfrak{m}}(M)^{\vee} \cong {\rm Ext}_{S}^{n-i}(M, \omega_{s}).$$
\end{theorem}

Let us also recall the notion of $k$-Buchsbaum modules  from \cite{FV}.
\begin{definition}
Let $M$ be a $\mathbb{Z}^{n}$-graded $S$-module. We say that $M$ is \emph{$k$-Buchsbaum} if $k$ is the minimal non-negative integer satisfying $\mathfrak{m}^{k} \cdot H_{\mathfrak{m}}^{i}(M) = 0$ for $i < {\rm dim} \, M$.
\end{definition}
\begin{remark}
\
\begin{enumerate}
\item[(i)]
Being a module $M$ with $0$-Buchsbaum property is equivalent to the fact that $M$ is Cohen-Macaulay.
\item[(ii)]
The Buchsbaum property implies $1$-Buchsbaumness, see \cite[Corollary 2.4]{StuckradVogel}.
\item[(iii)]
In general, being $1$-Buchsbaum does not imply the Buchsbaum property, see \cite[Example 2.5]{StuckradVogel} for a nice geometrical counterexample. However, in some cases, like squarefree modules, the notion of $1$-Buchsbaum and Buchsbaum modules coincides.
\end{enumerate}
\end{remark}
We say that a $\mathbb{Z}^{n}$-graded $S$-module $M$ is \emph{degreewise finite} if $M_{\textbf{b}}$ is a finitely dimensional $\Bbbk$-vector space for every $\textbf{b} \in \mathbb{Z}^{n}$. 
\begin{lemma}
\label{prop:annihilation-dual}
Let $M$ be a degreewise finite $\mathbb{Z}^n$-graded $S$-module and $N=M^{\vee}$. Then $\mathfrak{m}^{k} M = 0$ if and only if $\mathfrak{m}^{k}N = 0$.

\end{lemma}
\begin{proof}
If $\mathfrak{m}^{k} M = 0$, then by our definition of the multiplication map on the dual modules we have $\mathfrak{m}^{k}N = 0$. The revers implication follows from the fact that $(M^\vee)^\vee = M$.
\end{proof}

\begin{corollary}
\label{cor:kBuchsbaum-Ext}
Let $M$ be a finitely generated $\mathbb{Z}^n$-graded $S$-module.
Then the following statements are equivalent:
\begin{enumerate}
\item[(i)]
$M$ is $k$-Buchsbaum;
\item[(ii)]
 $\mathfrak{m}^{k} \cdot {\rm Ext}_{S}^{n-i}(M, \omega_{s})=0$ for every $i < {\rm dim} \, M$.
\end{enumerate}
\end{corollary}
\begin{proof}
This follows from Lemma \ref{prop:annihilation-dual} and Theorem \ref{local:dual}.
\end{proof}

Now we pass to squarefree modules. Let us recall basics on them, following Yanagawa's approach.
\begin{definition}
We say that a $\mathbb{Z}^{n}$-graded $S$-module $M$ is squarefree if the following conditions are satisfied:
\begin{enumerate}
    \item[(i)] $M$ is finitely generated;
    \item[(ii)] $M = \oplus_{\textbf{a} \in \mathbb{N}^{n}} M_{\textbf{a}}$;
    \item[(iii)] the map
$M_{\textbf{a}} \ni y \mapsto x_{i}y \in M_{\textbf{a}+\varepsilon_{i}}$
is bijective for all $\textbf{a} \in \mathbb{N}^{n}$ and $i \in {\rm supp}(\textbf{a})$.
\end{enumerate}
\end{definition}
For example, a Stanley-Reisner ring $\Bbbk[\triangle]$ is a squarefree $S$-module. Moreover, if $M$ and $N$ are squarefree $S$-modules and $f : M \rightarrow N$ is a degree-preserving map, then both ${\rm ker}(f)$ and ${\rm coker}(f)$ are squarefree -- see Yanagawa's paper \cite{Yanagawa} for details. One immediately sees:

\begin{lemma}
\label{lemma:squarefree-dual}
Let $M$ be a finitely generated $\mathbb{Z}^n$-graded $S$-module and $N=M^{\vee}$.
Then the following statements are equivalent:
\begin{enumerate}
\item[(i)]
$M$ is squarefree;
\item[(ii)]
The map
$N_{\textbf{b}-\varepsilon_{i}} \ni y \mapsto x_{i}y \in N_{\textbf{b}}$
is bijective for all $\textbf{b} \in -\mathbb{N}^{n}$ and $i \in {\rm supp}(\textbf{b})$, and $N_{\textbf{b}} = 0$ provided that $\textbf{b} \not\in -\mathbb{N}^{n}$.
\end{enumerate}
\end{lemma}

\begin{definition}[Dual squarefree module]
Let $M$ be a $\mathbb{Z}^{n}$-graded $S$-module. We say that $M$ is a \emph{dual squarefree module} if the following conditions are satisfied:
\begin{enumerate}
    \item[(i)] $M$ is Artinian;
    \item[(ii)] $N_{\textbf{b}} = 0$ provided that $\textbf{b} \not\in -\mathbb{N}^{n}$;
    \item[(iii)] The multiplication map 
    $N_{\textbf{b}-\varepsilon_{i}} \ni y \mapsto x_{i}y \in N_{\textbf{b}}$
is bijective for all $\textbf{b} \in -\mathbb{N}^{n}$ and $i \in {\rm supp}(\textbf{b})$.
\end{enumerate}
\end{definition}
\begin{corollary}
\label{cor:localcoh-sq}
If $M$ is a squarefree $S$-module, then $H^{i}_{\mathfrak{m}}(M)$ are dual squarefree $S$-modules for all $i$.
\end{corollary}
\begin{proof}
Observe that if $M$ is a squarefree $S$-module, then also ${\rm Ext}^{i}_{S}(M,\omega_{S})$ is squarefree for all $i$, which follows from Yanagawa's paper \cite{Yanagawa}. The claim of the corollary follows from Lemma \ref{lemma:squarefree-dual} and Theorem \ref{local:dual}.
\end{proof}

\begin{lemma}
\label{lemma:supporthelper}
Let $M$ be a $\mathbb{Z}^{n}$-graded $S$-module. 
\begin{enumerate}
\item[(i)]
If $M$ is either squarefree or dual squarefree and  $\dim_{\Bbbk} M<\infty$, then $M_{\textbf{b}}=0$ for any $\textbf{b}\in \mathbb{Z}^{n}\setminus \{0\}$.
\item[(ii)]
If $M$ is dual squarefree, then $\mathfrak{m}^{1} M=0$ if and only if $M_{\textbf{b}}=0$ for any $\textbf{b} \in \mathbb{Z}^{n}\setminus \{0\}$.
\end{enumerate}
\end{lemma}

The first main result of this paper connects various Buchsbaum properties of squarefree modules as follows:

\begin{theorem}
\label{thm:buchsbaum-characterizations}
Let $M$ be a squarefree module over $S=\Bbbk[x_{1}, ..., x_{n}]$ and $d=\dim M$. 
Then the following statements are equivalent:
\begin{enumerate}
\item[(i)]
$M$ is Buchsbaum;
\item[(ii)]
$M$ is $1$-Buchsbaum, or equivalently  $\mathfrak{m}^{1} \cdot {\rm Ext}_{S}^{n-i}(M, \omega_{s})=0$ for every $i < d$;
\item[(iii)]
$M$ is $k$-Buchsbaum for every $k\geq 1$, or equivalently  $\mathfrak{m}^{k} \cdot {\rm Ext}_{S}^{n-i}(M, \omega_{s})=0$ for every $i < d$ and  $k\geq 1$;
\item[(iv)]
$M$ is $k$-Buchsbaum for some $k\geq 1$, or equivalently  $\mathfrak{m}^{k} \cdot {\rm Ext}_{S}^{n-i}(M, \omega_{s})=0$ for every $i < d$ and for some $k\geq 1$.
\end{enumerate}
\end{theorem}
\begin{proof}
The dual statements follow all from Corollary \ref{cor:kBuchsbaum-Ext} and in the following we concentrate on the equivalence of the first parts of each statement (i) to (iv).

(i) $\Rightarrow$ (ii): If $M$ is Buchsbaum, then by \cite[Corollary 2.7, (a) $\Rightarrow$ (d)]{Yanagawa} we have
\[
H^{i}_{\mathfrak{m}}(M)=H^{i}_{\mathfrak{m}}(M)_0 \text{ for any } i=0,\dots,d-1
\]
and thus $\mathfrak{m}^{1}H^{i}_{\mathfrak{m}}(M)=0$ by Lemma \ref{lemma:supporthelper} (ii). Hence, $M$ is $1$-Buchsbaum.

(ii) $\Rightarrow$ (i):
If (ii) holds, then by Corollary \ref{cor:localcoh-sq} and Lemma \ref{lemma:supporthelper} (ii)
we have that 
\[
H^{i}_{\mathfrak{m}}(M)=H^{i}_{\mathfrak{m}}(M)_0 \text{ for any } i=0,\dots,d-1.
\]
Then (i) follows from \cite[Corollary 2.7, (d) $\Rightarrow$ (a)]{Yanagawa}.

(ii) $\Rightarrow$ (iii), (iii) $\Rightarrow$ (iv): These implications are trivially true by definitions.

(iv) $\Rightarrow$ (ii): Assume that $M$ is $k$-Buchsbaum for some $k\geq 1$.
Next, we assume that for some $i\in \{0,\dots,d-1\}$ there exists a vector $\textbf{b}\in -\mathbb{N}^n$
such that 
\[
H^{i}_{\mathfrak{m}}(M)_{\textbf{b}} \neq 0.
\]
Choose $j\in {\rm supp}(\textbf{b})$. Then it follows that
$H^{i}_{\mathfrak{m}}(M)_{\textbf{b}-k \varepsilon_j} \neq 0,$
by Corollary \ref{cor:localcoh-sq}. This yields the contradiction
\[
\mathfrak{m}^{k}
H^{i}_{\mathfrak{m}}(M)_{\textbf{b}-k \varepsilon_j} \neq 0.
\]
Hence, $H^{i}_{\mathfrak{m}}(M)=H^{i}_{\mathfrak{m}}(M)_0$ and (ii) follows from 
Lemma \ref{lemma:supporthelper} (ii).

\end{proof}

\section{Buchsbaumness for Levi graphs of line arrangements}

In the main results of this section we study in particular the relationship of various ring properties considered so far for edge rings of $d$-arrangements. Our proof is based on three steps. First of all, we are going to use a result due to Herzog and Hibi which provides a combinatorial description of edge ideals associated with bipartite graphs, and a result due to Cook II and Nagel about the equivalence of Cohen-Macaulay and Buchsbaum algebras associated with bipartite graphs. We sum up these results in the forthcoming theorem.
\begin{theorem}
Let $G$ be a bipartite graph with the partition $V_{1} = \{x_{1}, ..., x_{n}\}$ and $V_{2} = \{y_{1}, ..., y_{n'}\}$. Then the following conditions for $S/ I(G)$ are equivalent:
\begin{itemize}
    \item[(i)]  $S/ I(G)$ is Cohen-Macaulay;
    \item[(ii)] $S/ I(G)$ is Buchsbaum for non-complete bipartite graph $G$;
    \item[(iii)] $S/ I(G)$ is $k$-Buchsbaum for some $k\geq 0$;
    \item[(iv)]  $n=n'$ and there exists a re-ordering of the sets of vertices $V_{1},V_{2}$ such that
\begin{itemize}
    \item[a)] $x_{i}y_{i} \in E$ for all $i$,
    \item[b)] if $x_{i}y_{j} \in E$, then $i\leq j$,
    \item[c)] if $x_{i}y_{j}$ and $x_{j}y_{k}$ are in $E$, then $x_{i}y_{k} \in E$;
\end{itemize}
\item[(v)] $G$ has a cross-free pure order.
\end{itemize}
\end{theorem}
\begin{remark}
We say that the partitioning and ordering of vertices in $G$ satisfying $a)$ and $b)$ in $(iv)$ above is a \emph{pure order} of $G$. Furthermore, we say that a pure order has a \emph{cross} if, for some $i \neq j$, $x_{i}y_{j}$ and $x_{j}y_{i}$ are edges of $G$, otherwise we say the order is cross-free.
\end{remark}
\begin{proof}
(i) $\Leftrightarrow$ (ii) follows from \cite{CookNagel}, (i) $\Leftrightarrow$ (iv) is obtained from \cite{HerzogHibi}, and (iv) $\Leftrightarrow$ (v) proven again in \cite{CookNagel}. We need to show (ii) $\Leftrightarrow$ (iii), but this follows from the fact that $S/I(G)$ is squarefree and then we can apply Theorem  \ref{thm:buchsbaum-characterizations} (ii) $\Leftrightarrow$ (iii).
\end{proof}
\begin{lemma}
In the setting of the above characterization, a necessary condition that the edge ideal determined by the Levi graph $G$ of a line arrangement $\mathcal{L}$ is Cohen-Macaulay is $$s = \sum_{r\geq 2} t_{r} = k.$$
\end{lemma}
 A classical result due to de-Bruijn and Erd\H{o}s \cite{deBruijnErdos} provides a complete classification of such line arrangements with $s=k$.
\begin{theorem}[de-Bruijn - Erd\H{o}s]
Let $\mathcal{L} \subset \mathbb{P}^{2}_{\Bbbk}$ be an arrangement of $k\geq 3$ lines in the plane such that $t_{k}=0$, where $\Bbbk$ is an arbitrary field. Then $s \geq k$ and the equality holds if and only if $\mathcal{L}$ is either
\begin{itemize}
\item a Hirzebruch quasi-pencil consisting such that $t_{k-1}=1$ and $t_{2}=k-1$, or
\item a finite projective plane arrangement consisting of $q^{2}+q+1$ points and $q^{2}+q+1$ lines, where $q=p^{n}$ for some prime number $p \in\mathbb{Z}_{\geq 2}$.
\end{itemize}
\end{theorem}

Based on de-Bruijn and Erd\H{o}s theorem we can formulate the following.
\begin{theorem}
\label{theorema}
Let $\Bbbk$ be any field and let $\mathcal{L}\subset \mathbb{P}^{2}_{\Bbbk}$ be an arrangement of $k\geq 3$ lines with $s$ intersection points. Denote by $I(G)$ the edge ideal determined by Levi graph associated with $\mathcal{L}$. Then $\Bbbk[x_{1}, ..., x_{s},y_{1}, ..., y_{k}]/I(G)$ is never Cohen-Macaulay.
\end{theorem}

\begin{proof}
We are going to show that both Hirzebruch quasi-pencils and finite projective plane do not satisfy Herzog-Hibi's criterion. We start with the case when $\mathcal{L} = \{\ell_{1}, ..., \ell_{k}\}$ is a quasi-pencil.
Assume that $V_{1} = \{x_{1}, ..., x_{k}\}$ corresponds to the set of intersection points of $\mathcal{L}$ and $V_{2} = \{y_{1}, ..., y_{k}\}$ corresponds to the set of lines in $\mathcal{L}$. Take any ordering satisfying the first two conditions in Herzog-Hibi's criterion. Without lost of generality, we may assume the following conditions (up to relabelling of vertices):
\begin{itemize}
    \item $x_{1}$ corresponds to the point $P_{1}$ of multiplicity $d-1$;
    \item $y_{1}, ..., y_{k-1}$ are the elements corresponding lines $\ell_{1}, ..., \ell_{k-1}$ passing through the point $P_{1}$ and $P_{1} \not\in \ell_{k}$;
    \item $y_{k}$ corresponds to $\ell_{k}$;
    \item $P_{i}$ is a double intersection point and $x_{i}$ is corresponding element to $P_{i}$;
    \item $y_{i}, y_{k}$ are the elements corresponding to lines intersecting at $P_{i}$.
\end{itemize}
Then obviously $x_{1}y_{i}$ and $x_{i}y_{k}$ are the edges of $G$, but $x_{1}y_{k}$ does not correspond to any edge in the Levi graph $G$ since $P_{1}$ is not incident with line $\ell_{k}$. This shows that the associated algebra cannot be Cohen-Macaulay. 

Next we consider the case of a finite projective plane $\mathcal{F}_{q}$ which has exactly $q^{2}+q+1$ points and $q^{2}+q+1$ lines. 
Assume that $V_{1} = \{x_{1}, ..., x_{q^{2}+q+1}\}$ corresponds to the set of intersection points of $\mathcal{L}$ and $V_{2} = \{y_{1}, ..., y_{q^{2}+q+1}\}$ corresponds to the set of lines in the arrangement. Take any ordering satisfying the first two conditions in Herzog-Hibi's criterion. We are going to show that the third condition of the mentioned criterion is not satisfied. To this end, consider the point $(1:1:1)$ and  we denote by $x_{i}$ the corresponding element. Then take the line $(q-1)x + z = 0$ and the corresponding element $y_{j}$, the point $(q-1:0:1)$ and the corresponding element $x_{j}$, and we take finally the line $y=0$ with the corresponding element $y_{k}$. Observe that $x_{i}y_{j}$ and $x_{j}y_{k}$ are the edges of the Levi graph $G$, but $x_{i}y_{k}$ is not any edge of $G$ since the point $(1,1,1)$ is not incident with the line $y=0$. This concludes the proof.
\end{proof}
\begin{example}
Unfortunately, we do not know how to extend the above result to the case of any arbitrary $d$-arrangement with $d\geq 2$. One can show that if $\mathcal{C}$ is a $d$-arrangement with $d\geq 2$, $k\geq 3$ and $t_{k}=0$, then by \cite[Lemma 4.3]{PSZR} we have
$$s = \sum_{r\geq 2} t_{r} \geq k,$$
but we do not have a global description of $d$-arrangements with $s=k$.  If we restrict our attention to $2$-arrangements, we can easily construct an arrangement consisting of $6$ conics and $6$ intersection points of multiplicity $5$ -- just take $6$ general points and all smooth conics determined by subsets consisting of $5$ distinct points. We call such a configuration  a \emph{$(6_{5},6_{5})$-symmetric point-conic configuration}. Let us consider the edge ideal associated with the above arrangement
$$I(G) = \langle x_{1}y_{1}, x_{1}y_{2}, x_{1}y_{3}, x_{1}y_{4}, x_{1}y_{5}, x_{2}y_{1},x_{2}y_{2},x_{2}y_{3}, x_{2}y_{4},x_{2}y_{6},$$
$$ x_{3}y_{1},x_{3}y_{2},x_{3}y_{3},x_{3}y_{5},x_{3}y_{6}, x_{4}y_{1},x_{4}y_{2},x_{4}y_{4},x_{4}y_{5},x_{4}y_{6},$$
$$x_{5}y_{1},x_{5}y_{3},x_{5}y_{4},x_{5}y_{5},x_{5}y_{6}, x_{6}y_{2},x_{6}y_{3},x_{6}y_{4},x_{6}y_{5},x_{6}y_{6} \,\rangle.$$
Using \verb{Singular{ one can compute firstly the minimal graded free resolution of the algebra $$\Bbbk[x_{1}, ..., x_{6}, y_{1}, ..., y_{6}]/I(G).$$

The Betti diagram has the following form:
\begin{center}
\begin{verbatim}
           0     1     2     3     4     5     6     7     8     9    10
----------------------------------------------------------------------------
    0:     1     -     -     -     -     -     -     -     -     -     -
    1:     -    30   120   210   180    62     -     -     -     -     -
    2:     -     -    15   120   400   720   765   500   204    48     5
----------------------------------------------------------------------------
 total:    1    30   135   330   580   782   765   500   204    48     5

\end{verbatim}
\end{center}
Based on this we can conclude that the regularity of the algebra is equal to $2$, and the projective dimension is equal to $10$. Moreover, we check directly that the algebra  is not Cohen-Macaulay. 
\end{example}
\begin{problem}
Classify all $d$-arrangements with $d\geq 2$ such that $s=k$.
\end{problem}
On the other hand, it is natural to ask whether there exists a $d$-arrangement with $d\geq 2$ such that its algebra is Cohen-Macaulay. We will come back to this question in the forthcoming section -- it turns out that such a $d$-arrangement does not exist.

\section{Sequentially Cohen-Macaulay algebras and $d$-arrangements}
As we saw in the previous sections, edge ideals associated with line arrangements are neither Cohen-Macaulay nor $k$-Buchsbaum for some $k\geq 1$. It is natural to wonder how this situation looks like if we focus on generalizations of Cohen-Macaulay rings. Here we focus on sequentially Cohen-Macaulay rings. 
\begin{definition}
Let $S = \Bbbk[x_{1}, ..., x_{n}]$. A graded $S$-module $M$ is called \emph{sequentially Cohen-Macaulay} over $\Bbbk$ if there exists a finite filtration of graded $S$-modules
$$0=M_{0} \subset M_{1} \subset ... \subset M_{r} = M$$
such that each $M_{i}/M_{i-1}$ is Cohen-Macaulay and the Krull dimensions of the quotients are increasing
$${\rm dim}\, (M_{1}/M_{0}) < {\rm dim}(M_{2}/M_{1}) < ... < {\rm dim}(M_{r}/M_{r-1}).$$
\end{definition}
\begin{definition}
Let $G$ be a graph whose independence complex is $\triangle_{G}$. 
\begin{enumerate}
\item[(i)]
We say that $G$ is \emph{shellable graph} if $\triangle_{G}$ is a shellable simplicial complex.
\item[(ii)]
The graph $G$ is called \emph{sequentially Cohen-Macaulay} if the algebra $S/I(G)$ is sequentially Cohen-Macaulay.
\end{enumerate}
\end{definition}
In this section we are going to provide a complete classification of sequentially Cohen-Macaulay algebras associated with edge ideals of Levi graphs. In order to do so, we need the following two results which come from a paper by van Tuyl and Villarreal \cite{vaTuylVillarreal}.
\begin{lemma}
\label{shell}
Let $G$ be a bipartite graph with bipartition $\{x_{1},..., x_{m}\}$, $\{y_{1},..., y_{n}\}$. If $G$
is shellable and $G$ has no isolated vertices, then there is $v \in V(G)$ with ${\rm deg}(v) = 1$.
\end{lemma}
\begin{theorem}
Let $G$ be a bipartite graph. Then $G$ is shellable if and only if $G$ is sequentially Cohen-Macaulay.
\end{theorem}
Now we are ready to provide our next classification and main result of our work.
\begin{theorem}
\label{theoremb}
Let $\mathcal{C}$ be a $d$-arrangement of $k\geq 3$ curves in $\mathbb{P}^{2}_{\mathbb{C}}$ with $s$ intersection points and let $I(G)$ be the associated edge ideal determined by the Levi graph of $\mathcal{C}$. Then the following conditions are equivalent:
\begin{itemize}
    \item[(i)] $\Bbbk[x_{1}, ..., x_{s},y_{1}, ..., y_{k}]/I(G)$ is sequentially Cohen-Macaulay;
    \item[(ii)] $\mathcal{C}$ is a pencil of $k$ lines in the plane.
\end{itemize}
\end{theorem}
\begin{proof}
First of all, it is clear that if $\mathcal{L}$ is a pencil of $k$ lines, then the associated Levi graph is a tree with one root $x_{1}$ corresponding to the $d$-fold point $P$ and $k$ leafs $y_{1}, ..., y_{k}$ corresponding to lines passing through $P$. Since the graph $K_{1,k}$ is shellable, then the associated monomial algebra $\Bbbk[x_{1},y_{1}, ...,y_{n}]/I(G)$ is sequentially Cohen-Macaulay. 

From now on we assume that $\mathcal{C}$ is a line arrangement with $t_{k}=0$ or a $d$-arrangement with $d\geq 2$ and assume that $\Bbbk[x_{1}, ..., x_{s},y_{1}, ..., y_{k}]/I(G)$ is sequentially Cohen-Macaulay.

We are going to apply Lemma \ref{shell}, i.e., our aim is to show that in that setting the Levi graph of $\mathcal{C}$ has vertices of degree greater or equal to $2$, so $G$ is not shellable. We assume that $V_{1} = \{x_{1}, ..., x_{s}\}$ corresponds to the intersection points of $\mathcal{C}$ and $V_{2} = \{y_{1}, ..., y_{k}\}$ corresponds to the curves in $\mathcal{C}$.
If $\mathcal{C}$ is a line arrangement with $t_{k}=0$, then each intersection point has multiplicity at least $2$, so ${\rm deg}(x_{i}) \geq 2$. Since $t_{k} = 0$, then by using (\ref{equ2}) with $d=1$ we see that on each line we have at least $2$ intersection points from ${\rm Sing}(\mathcal{C})$ which means that ${\rm deg}(y_{j}) \geq 2$, and it completes the proof for case of line arrangements.

Suppose that $\mathcal{C}$ is a $d$-arrangement with $d\geq 2$. Then each intersection point has multiplicity at least $2$ and this gives ${\rm deg}(x_{i}) \geq 2$. Moreover, using the combinatorial count (\ref{equ2}), we can observe that ${\rm deg}(y_{j}) \geq d^{2}$. This completes the proof.
\end{proof}
As a corollary, we obtain the following classification result.
\begin{corollary}
Let $\mathcal{C}\subset \mathbb{P}^{2}_{\mathbb{C}}$ be a $d$-arrangement of $k\geq 3$ curves with $d\geq 2$ and $s$ intersection points. Denote by $I(G)$ the edge ideal determined by Levi graph associated with $\mathcal{C}$. Then $\Bbbk[x_{1}, ..., x_{s},y_{1}, ..., y_{k}]/I(G)$ is never Cohen-Macaulay.
\end{corollary}
\begin{proof}
Let us recall that for algebras being Cohen-Macaulay implies being sequentially Cohen-Macaulay. Since for $d$-arrangements with $d\geq 2$ the associated Levi graphs are never shellable, then the associated algebras $\Bbbk[x_{1}, ..., x_{s},y_{1}, ..., y_{k}]/I(G)$ are not sequentially Cohen-Macaulay. This completes the proof since pencil of lines are not $d$-arrangements with $d\geq 2$.
\end{proof}
\section{Bounds on the projective dimension}
This section is motivated by the following result due to Dao, Huneke, and Schweig \cite[Corollary 5.4]{DHS}.
\begin{proposition}
Let $G$ be a graph on $n$ vertices, and assume that $m$ is the maximal
degree of any vertex. Then 
$${\rm pd}(S/I(G)) \leq n \cdot \bigg(1 - \frac{1}{2m} \bigg).$$
\end{proposition}
Using the above result we are going to provide a reasonable upper-bound for the projective dimension of an edge ideals associated with a $d$-arrangement.
\begin{proposition}
\label{cor:projdim}
Let $\mathcal{C}$ be a $d$-arrangement of $k\geq 3$ curves having $s$ intersection points such that $t_{k}=0$. Let $I(G)$ be the associated edge ideal in $S=\Bbbk[x_{1}, ..., x_{s},y_{1}, ...,y_{k}]$. Then
$$\bigg\lceil \frac{s+k}{2} \bigg\rceil \leq  {\rm pd} (S/I(G)) \leq (s+k) \bigg(1 - \frac{1}{2d^{2}(k-1)}\bigg).$$ 

\end{proposition}
\begin{proof}
If $p \in {\rm Sing}(\mathcal{C})$, then its multiplicity can be at most $k-1$ since $t_{k}=0$. On the other hand, using equation (\ref{equ2}), one can have at most $d^{2}(k-1)$ intersection points on each curves. This gives
for the maximal degree of any vertex
$m \leq {\rm max}\, \{k-1, d^{2}(k-1)\} = d^{2}(k-1).$
The stated lower bound follows directly from \cite[Lemma 3.1]{EH}, and this completes the proof.
\end{proof}
\begin{example}[Projective dimension versus regularity]
Let us consider an arrangement of $5$ general lines in the plane, i.e., such an arrangement has only double points as the intersections. We have altogether $10$ such points. Let us denote by $x_{i}$ the variable corresponding to the intersection points and by $y_{j}$ the variables corresponding to lines. The edge ideal $I(G)$ in  $S=\Bbbk[x_{1}, ..., x_{10},y_{1}, ..., y_{5}]$ of the Levi graph $G$ has the following presentation
$$I(G) = \langle x_1 y_1 , x_1 y_2 , x_2 y_2 , x_2 y_3 , x_3 y_2 , x_3 y_4 , x_4 y_2 , x_4 y_5 , x_5 y_4 , x_5 y_5 , x_6 y_1 , x_6 y_4 , x_7 y_3 , x_7 y_4 ,$$
$$x_8 y_5 , x_8 y_3 , x_9 y_1 , x_9 y_3 , x_{10} y_1 , x_{10} y_5 \rangle.$$

Using \verb{Singular{ we can compute the minimal graded free resolution of the algebra $S/I(G)$. The Betti diagram has the following form:
\begin{center}
\begin{verbatim}
         0     1     2     3     4     5     6     7     8     9    10    11
----------------------------------------------------------------------------
    0:   1     -     -     -     -     -     -     -     -     -     -     -
    1:   -    20    40    20     5     -     -     -     -     -     -     -
    2:   -     -    90   360   515   330   100    10     -     -     -     -
    3:   -     -     -    80   470  1135  1370   897   295    40     -     -
    4:   -     -     -     -     5    35   105   185   205   120    35     4
----------------------------------------------------------------------------
total:   1    20   130   460   995  1500  1575  1092   500   160    35     4

\end{verbatim}
\end{center}

We can easily see that the regularity of the algebra $S/I(G)$ is equal to $4$ and the projective dimension is is equal to $11$. Notice that Corollary \ref{cor:projdim} tells us that the projective dimension of $S/I(G)$ sits in the interval $[8, 13.125]$. 
\end{example}
Based on many similar examples (i.e., arrangements of $k \geq 3$ general lines) we can observe that the regularity of the edge ideals is smaller than the number of lines (and this follows from our general bound from the previous section), but the projective dimension grows as fast as the number of double intersection points which makes the Betti diagram wide-rectangular. It is quite natural to believe that the lower bound for the projective dimension is related to the number of intersection points of a given arrangements. In this context, we propose:
\begin{problem}
Let $\mathcal{L}$ be an arrangement of $k \geq 3$ general lines in the plane and denote by $I(G)$ the associated edge ideal. Is it true that $${\rm pd}\, (S/I(G)) \geq \binom{k}{2} \,?$$
\end{problem}

\section{Castelnuovo-Mumford regularity for powers of edge ideals of $d$-arrangements}
After studying classical algebraic properties of modules, we are going back to homological properties.
Now we pass towards a general regularity bound. Let us consider:
\begin{example}
Consider a Hirzebruch quasi-pencil $\mathcal{L}_{k}$ of $k$ lines with the associated Levi graph $G$. As usually, $V_{1} = \{x_{1}, ..., x_{k}\}$ denotes the vertices corresponding to the intersection points and $V_{2} = \{y_{1}, ..., y_{k}\}$ corresponds to the lines in $\mathcal{L}_{k}$. The associated edge ideal in $S_{k} := \Bbbk[x_{1}, ..., x_{k},y_{1}, ..., y_{k}]$ can be presented as follows:
$$I(G)_{k} = \bigg\langle x_{1}y_{1}, ..., x_{k}y_{k}, x_{1}y_{2}, x_{1}y_{3}, ..., x_{1}y_{k-1},x_{2}y_{k},x_{3}y_{k}, ..., x_{k-1}y_{k},x_{k}y_{1} \bigg\rangle.$$
By the previous results, we know that the algebra $A_{k} = S_{k}/I(G)_{k}$ is neither Cohen-Macaulay nor sequentially Cohen-Macaulay. However, it is still interesting to check other homological properties of $A_{k}$. First of all, we want to find a reasonable bound on the Castelnuovo-Mumford regularity of $A_{k}$. There are some results providing the upper-bounds for regularity of edge ideals which use the combinatorial description of graphs. 

H\'a and van Tuyl in \cite{HavanTuyl} showed that for any graph $G$ the regularity of edge ideal $I(G)$ is bounded from above by
\[
{\rm reg}(I(G)) \leq \nu(G)+1,
\]
where $\nu(G)$ is the matching number of a graph $G$ -- this is the maximum number of pairwise disjoint edges. The matching number of the Levi graph associated with $\mathcal{L}_{k}$ is equal to $k$ and this is easy to see. This allows us to conclude that for every $k \geq 3$ one has
$${\rm reg}(A_{k}) \leq k.$$
\end{example}
This example has a global manifestation that is valid for all $d$-arrangements. In order to present our main results, which are corollaries from what we will see in a moment, we need the following direct application of Hall's theorem.

\begin{proposition}
\label{Hall}
Let $\mathcal{C}$ be a $d$-arrangement with $d\geq 1$ and $k\geq 3$. Assume that $t_{k}=0$ and denote by $G$ the associated Levi graph with the bipartition $\{x_{1}, ..., x_{s}\}$, $\{y_{1}, ..., y_{k}\}$. Then $G$ has matching of size $k$.
\end{proposition}
\begin{proof}
According to Hall's theorem, we need to check whether for any subset $S \subset \{y_{1}, ..., y_{k}\}$ we have $|N(S)| \geq |S|$, where $$N(S) = \{x_{i} \, : \, \exists \, y_{j} \in S \text{ with } \{x_{i},y_{j}\} \in E\}.$$ Since $G$ is connected and for each $y_{i}$ we have that ${\rm deg}(y_{i}) \geq d^{2}+1$, then
for every $S \subset \{y_{1}, ..., y_{k}\}$ with $|S|= \ell$ one has
$$|N(S)| \geq (d^{2}+1)\cdot \ell \geq 2\ell,$$
which completes the proof.
\end{proof}

\begin{corollary}[A global regularity bound]
\label{regupp}
Let $\mathcal{C}$ be a $d$-arrangement of $k\geq 3$ curves with $d\geq 1$ in $\mathbb{P}^{2}_{\mathbb{C}}$ with $s$ intersection points, and assume that $t_{k}=0$. Let $I(G)$ be the associated edge ideal determined by the Levi graph $G$ of $\mathcal{C}$. Then
$${\rm reg}(I(G)) \leq k + 1.$$
\end{corollary}
\begin{proof}
Observe that Proposition \ref{Hall} provides us the maximal value of the matching number $\nu(G)$ of the Levi graph $G$. First of all, Levi graphs $G$ satisfy the following properties on the degrees of vertices:
$$\forall \, i \in \{1, ..., s\},j \in \{1, ..., k\} \, : \, {\rm deg}(x_{i}) \geq 2 \text{ and  }{\rm deg}(y_{j}) \geq 2.$$
Now the maximality of our choice follows from a simple incidence property -- if we add another edge, then such an edge cannot be disjoint from the matching edges since each vertex of $G$ has degree greater of equal to $2$. Now we apply \cite{HavanTuyl} to conclude the proof.
\end{proof}
It is natural to wonder whether our upper bound is close to the real value of the regularity. 
\begin{example}
Consider the case of a Hirzebruch quasi-pencil with $k=4$ and observe that our global upper bound gives that the regularity of $I(G)$ is bounded by $5$. Using \verb{Singular{ we can compute the minimal graded free resolution and the Betti diagram has the following shape:
\begin{center}
\begin{verbatim}
                       0     1     2     3     4     5     6
           ---------------------------------------------------
                0:     1     -     -     -     -     -     -
                1:     -     9    12     2     -     -     -
                2:     -     -     9    24    18     6     1
                3:     -     -     -     1     2     1     -
           ---------------------------------------------------
            total:     1     9    21    27    20     7     1    
\end{verbatim}
\end{center}
This means that the regularity of the edge ideal $I(G)$ is equal to $4$.
\end{example}
\begin{example}
If $\mathcal{L}$ is a pencil of $k$ lines in the plane, then by Fr\"oberg's characterization \cite{froberg} we have
$${\rm reg}(I(G)) = 2$$
showing that the upper bound by H\'a and van Tuyl is sharp. This can be also explained in a slightly different way since $I(G)$ is sequentially Cohen-Macaulay and then by \cite{vanTuyl1} we always have 
$${\rm reg}(I(G)) = \nu(G)+1.$$
\end{example}
This observation motivates the following question.
\begin{problem}
Let $\mathcal{L}$ be a $d$-arrangement with $k\geq 3$ with the associated edge ideal $I(G)$. Is it true that ${\rm reg}(I(G)) = 2$ implies that $\mathcal{L}$ is a pencil of lines?
\end{problem}

Now we focus on powers of edge ideals. We start with revisiting some results related to ordinary powers of edge ideals. Let us recall that for a homogeneous ideal $I$ in a polynomial ring, ${\rm reg}(I^{m})$ is a linear function for $m\gg0$, it means that there exist integers $a,b,m_{0}$ such that
$${\rm reg}(I^{m}) = am+b \text{ for all } m\geq m_{0}.$$
This is a content of the result proved by Cutkosky, Herzog, and Trung in \cite{Cut}. It is known that $a$ is bounded above by the maximum of degrees of elements in a minimal set of generators of $I$, but $b$ and $m_{0}$ are in general difficult to detect. If we restrict our attention to an edge ideals $I(G)$ associated with any graph, then
$$\nu'(G) + 1 \leq {\rm reg}(I(G)) \leq \text{co-chord}(G)+1,$$
where $\nu'(G)$ is the induced matching number of $G$, and $\text{co-chord}(G)$ denotes the co-chordal cover number of $G$.
These inequalities follow from \cite{Katzman, Wood}.

Now we are going to focus on ${\rm reg}(I^{m})$ for some $m\geq 2$. Our observation comes from Proposition \ref{Hall} and results due to Jayanthan and Selvaraja \cite{Jayan}.

\begin{corollary}
Let $\mathcal{C}$ be a $d$-arrangement with $d\geq 1$, $k\geq 3$, and such that $k\leq s$. Denote by $G$ the associated Levi graph. Then
$${\rm reg}(I(G)^{q}) \leq 2q + k - 1.$$
\end{corollary}
\begin{proof}
This follows directly from \cite[Theorem 4.4]{Jayan} and the fact that $\text{co-chord}(G) \leq k$.
\end{proof}
\begin{remark}\
\begin{itemize}
    \item It is very natural to wonder whether there is a reasonable upper bound on the regularity of \textit{symbolic powers} of edge ideals. Let us recall that by \cite[Theorem 5.9]{Simis} for every bipartite graph $G$ and $m \in \mathbb{N}$ one has
$$I(G)^{(m)} = I(G)^{m},$$
which means that we have automatically obtained
$${\rm reg}(I(G)^{(q)}) \leq 2q + k-1.$$
\item It is natural to wonder whether we can find a reasonable upper-bound on the regularity of squarefree powers of edge ideals. This problem, in its whole generality, is open.
\end{itemize}
\end{remark}
We conclude this section by an observation regarding the Rees algebra of the edge ideal. Let us recall that for a given ideal $I$ the Rees algebra $\mathcal{R}(I) = \bigoplus_{i=0}^{\infty}I^{i}t^{i} \subset R[t]$.
\begin{corollary}
Let $\mathcal{C}$ be a $d$-arrangement with $k\geq 3$ and denote by $I(G)$ the associated edge ideal with the Levi graph $G$. Then 
$${\rm reg}(\mathcal{R}(I(G))) = k.$$
\end{corollary}
\begin{proof}
This follows from \cite[Theorem 1.2]{Cid} and Proposition \ref{Hall}.
\end{proof}

\section*{Acknowledgments}
The project was conducted when the first author was visiting twice the University of Osnabr\"uck in the framework of DAAD activities, namely the Personal Funding Programme \textit{Research Stays for University Academics and Scientists} in February 2021, and in July 2021 as the part of the Institutional Programme \textit{Ostpartnerschaften}. We would like to thank anonymous referees for comments that allowed to improve the paper.

\section*{Data availability} This is not applicable as the results presented in this manuscript rely on no external sources of data or code.

\vskip 0.5 cm

\bigskip
Piotr Pokora,
Department of Mathematics,
Pedagogical University of Krak\'ow,
Podchor\c a\.zych  2,
PL-30-084 Krak\'ow, Poland. \\
\nopagebreak
\textit{E-mail address:} \texttt{piotr.pokora@up.krakow.pl, piotrpkr@gmail.com}

\bigskip
Tim R\"omer,
Institute of Mathematics,
Research Unit Data Science,
Osnabr\"uck University, 
Albrechtstr. 28A
D-49076 Osnabrück, Germany. \\
\nopagebreak
\textit{E-mail address:} \texttt{troemer@uos.de}

\end{document}